\documentclass[12pt]{article}

\usepackage{amsthm}
\usepackage{amssymb}
\usepackage{amsmath}
\usepackage{mathtools}
\usepackage{xfrac}
\usepackage{url} 
\usepackage{hyperref}
\usepackage{slashbox}

\usepackage[ruled,vlined]{algorithm2e}
\usepackage{graphicx}

\usepackage{color}

\newcommand{\mc}[1]{\mathcal{#1}}

\theoremstyle{definition}

\theoremstyle{plain}
\newtheorem{theorem}{Theorem}[section]
\newtheorem{lemma}[theorem]{Lemma}
\newtheorem{claim}[theorem]{Claim}

\newtheorem{corollary}[theorem]{Corollary}

\title{Erd\H{o}s-Szekeres without induction}
\author{
Sergey Norin\thanks{Department of Mathematics and Statistics, McGill University. Email: {\tt snorin@math.mcgill.ca}. Supported by an NSERC grant 418520.}\and \and
Yelena Yuditsky\thanks{School of Computer Science, McGill University. Email: {\tt yuditskyl@gmail.com}.}}

\begin{document}

\maketitle

\begin{abstract}
Let $ES(n)$ be the minimal integer such that  any set of $ES(n)$ points in the plane in general position contains $n$ points in convex position. The problem of estimating $ES(n)$ was first formulated by Erd\H{o}s and Szekeres~\cite{ES35}, who proved that $ES(n) \leq \binom{2n-4}{n-2}+1$. The current best  upper bound, $\lim\sup_{n \to \infty} \frac{ES(n)}{\binom{2n-5}{n-2}}\le \frac{29}{32}$, is due to Vlachos~\cite{V15}. We improve this  to $$\lim\sup_{n \to \infty} \frac{ES(n)}{\binom{2n-5}{n-2}}\le \frac{7}{8}.$$ 


\end{abstract}

\section{Introduction}

The following problem  of Erd\H{o}s and Szekeres has attracted considerable attention over the years. 

\vskip 10pt
\noindent {\bf The Erd\H{o}s-Szekeres problem.}
For a positive integer $n \geq 3$, determine
the smallest integer $ES(n)$ such that any set of at least $ES(n)$ points in general
position in the plane contains $n$ points
that are the vertices of a convex $n$-gon.
\vskip 10pt

The first bounds on $ES(n)$ were given by  Erd\H{o}s and Szekeres \cite{ES35,ES60}, who have shown that  $$ 2^{n-2}+1 \le ES(n) \le \binom{2n-4} {n-2}+1.$$ They conjectured that the lower bound is tight. Their conjecture is verified for  $n \leq 6$. It is trivial for $n=3$; the proof for $n=4$ was given by Esther Klein and can be found in~\cite{ES35}; the proof for $n=5$, as stated in~\cite{ES35}, was first given by Makai and can be found in \cite{KKS70}; the proof for $n=6$ is due to  Szekeres and Peters~\cite{SP06}. 

For the upper bound, although it is very far from the lower bound, little improvement was made over the last eighty years. The first such improvement was given by Chung and Graham~\cite{CG98}, who improved the upper bound by one to $\binom{2n-4} {n-2}$. Shortly after, Kleitman and Pachter \cite{KP98} showed that $ES(n)\le \binom{2n-4}{n-2}+7-2n$. T\'{o}th and Valtr \cite{TV98,best-upper} have shown that $ES(n)\le {{2n-5} \choose {n-2} }+ 1$. Finally, very recently Vlachos~\cite{V15} have shown that $\lim\sup_{n \to \infty} \frac{ES(n)}{\binom{2n-5}{n-2}}\le \frac{29}{32}$. 

We  further improve the upper bound as follows.
\begin{theorem}\label{thm:main}
$$\lim\sup_{n \to \infty} \frac{ES(n)}{\binom{2n-5}{n-2}}\le \frac{7}{8}.$$
\end{theorem} 

Many variants of the Erd\H{o}s-Szekeres problem  were studied, and we refer the reader to surveys~ \cite{BK01,BMP04,MS00,best-upper}.

Our proof of Theorem~\ref{thm:main}, unlike the original proofs of all the bounds listed above, doesn't use induction, but is based on assigning a word in two letter alphabet to each point and analyzing the possible sets of resulting words. We describe the technique in Section~\ref{sec:sequences}.
We prove Theorem~\ref{thm:main} in Section~\ref{sec:points}.

\section{Encoding points}\label{sec:sequences}

In this section we introduce the concepts and prove the basic results underlying the proof of Theorem~\ref{thm:main}.

For a positive integer $m$, let $[m]:=\{1,2,\ldots,m\}$, and let $[m]^{(2)}$ denote the set of two element subsets of $[m]$.

Fix $f: [m]^{(2)}\rightarrow\mathbb{R}$. 
We say that a sequence $(q_1,q_2,\ldots,q_t)$ of elements of $[m]$, satisfying $q_1 <q_2< \ldots <q_t$ is a \textit{$t$-cup with respect to $f$} if $$f(q_1,q_2)\le f(q_2,q_3)\le \ldots \le f(q_{t-1},q_t),$$ and this sequence is a \textit{$t$-cap with respect to $f$} if $$f(q_1,q_2)\ge f(q_2,q_3)\ge \ldots \ge f(q_{t-1},q_t).$$
(We write $f(x,y)$ instead of $f(\{x,y\})$ for brevity.)
We say that a $t$-cup (or a $t$-cap) $(q_1,q_2,\ldots,q_t)$ \emph{starts} at $q_1$ and \emph{ends} in $q_t$.

We say that the function $f$ is \emph{$(k,l)$-free} if there exists no $k$-cup and no $l$-cap with respect to $f$. Let $ES'(k,l)$ denote the maximum $m$ such that there exist a $(k,l)$-free $f:[m]^{(2)} \to \mathbb{R}$. One of the main byproducts of the construction introduced in this section is a new proof of the following result of Erd\H{o}s and Szekeres. 

\begin{theorem}[\cite{ES35}]\label{thm:cupsandcaps}
$$ES'(k+2,l+2) \leq \binom{k+l}{k}.$$
\end{theorem} 

Our  goal is to assign words to every element of $[m]$ which partially encode the structure of cups and caps ending and starting at this element.
For every $i\in [m]$ we define four functions  $\alpha^i, \gamma^i: \mathbb{N} \to \mathbb{R}\cup \{ \pm \infty\}$ and $\beta^i,\delta^i: \mathbb{N} \to \mathbb{R}\cup \{\pm \infty\}$.
The values of these functions at  $t \in \mathbb{N}$ are defined as follows:
\begin{itemize}
\item Let  $\alpha^i(t)$ be the minimum $f(j,i)$ such that there exists a $(t+1)$-cup which ends with the pair $(j,i)$, i.e. a sequence $q_1<q_2<...<q_{t-1}<j<i$ such that 
$$f(q_1,q_2)\le...\le f(q_{t-1},j) \le f(j,i).$$
If there is no $(t+1)$-cup ending in $i$, let $\alpha^i(t)=+\infty$. 
\item Let $\beta^i(t)$ be the maximum $f(j,i)$  such that there exists a $(t+1)$-cap which ends with the pair $(j,i)$. If there is no such cap then $\beta^i(t)=-\infty$.
\item Let $\gamma^i(t)$ be the maximum  $f(i,j)$ such that there exists a $(t+1)$-cup which starts with the pair $(i,j)$, and if there is no such cup then $\gamma^i(c)=-\infty$.
\item Finally, let $\delta^i(t)$  be the minimum $f(i,j)$  such that there exists a $(t+1)$-cap which starts with the pair $(i,j)$, and if there is no such cap then $\delta^i(t)=+\infty$.
\end{itemize}
(The above functions encode the optimal values for further extending the cups and caps of given length which start or end at $i$.)

Assume now that the function $f$ is $(k+2,l+2)$-free. Let  ${\cal L}^{k,l}$ denote the set of all words of length $k+l$ in the alphabet $\{\alpha,\beta\}$, which use exactly 
$k$ symbols ``$\alpha$" and $l$ symbols ``$\beta$". We define a word $L_i \in {\cal L}^{k,l}$ for each $i \in [m]$ as follows.
Sort the multiset $( \alpha^i(1),\alpha^i(2),...,\alpha^i(k),\beta^i(1),\beta^i(2),...,\beta^i(l))$ in increasing order, using a convention that when the elements repeat, then the values of the function  $\alpha^i$ precede the values of the function $\beta^i$. The word  $L_i$ is obtained from the resulting sequence by replacing the values of $\alpha^i$ and $\beta^i$ by the symbols $\alpha$ and $\beta$ respectively. Figure~\ref{fig:example} provides an example.  (The first table is lists the values of the function $f$ and the second one the values of $\alpha^i$, $\beta^i$ and $L_i$.)

Let ${\cal R}^{k,l}$ denote the set of all words of length $k+l$ in the alphabet $\{\gamma,\delta\}$, which use exactly 
$k$ symbols ``$\gamma$" and $l$ symbols ``$\delta$". We define a word $R_i \in {\cal R}^{k,l}$, symmetrically to the word $L_i$, by sorting the values of $\gamma_i$ and $\delta_i$, using the convention that the $\delta$ symbols precede the  $\gamma$ symbols if the corresponding function values are equal.

\begin{figure}

\begin{center}
\begin{tabular}{| r  | r| r | r | r | r| } \hline
\backslashbox{$i$}{$j$} &1 & 2 & 3 & 4 & 5 \\ \hline
1 & &  2 & 4 &-2  & 1 \\ \hline
2 &2 &     & 6 &-6  &-1   \\ \hline
3 &4 & 6   &  & -9 & -5 \\ \hline
4 & -2& -6   &-9  & & 7\\ \hline
5 & 1& -1&  -5&  7&   \\ \hline
\end{tabular}
\end{center}

\begin{center}
    \begin{tabular}{ |r | r | r | r | r | c|}
    \hline
    $i$ &$\alpha^i(1)$ & $\alpha^i(2)$ & $\beta^i(1)$ & $\beta^i(2)$ & $L_i$ \\ \hline
    $1$ & $\infty$ & $\infty$ & $-\infty$ & $-\infty$ & $\beta\beta\alpha\alpha$  \\ \hline
    2 & $2$ & $\infty$ & $2$ & $-\infty$ & $\beta \alpha \beta \alpha$ \\ \hline
    3 & $4$ & $6$ & $6$ & $-\infty$ &  $\beta \alpha \alpha \beta$\\ \hline
    4 & $-9$  & $\infty$ & $-2$ & $-6$ &  $\alpha\beta\beta\alpha$\\ \hline
    5 & $-5$ & $7$ & $7$ & $-1$ & $\alpha \alpha \beta \beta$ \\ \hline
    \end{tabular}
\end{center}
\caption{\label{fig:example} An example of a $(4,4)$-free function and corresponding $L_i$'s.}
\end{figure}

We are now ready for the first result on $(k+2,l+2)$-free functions.

\begin{lemma} \label{lemma:a_b}
Let $f:[m]^{(2)} \to \mathbb{R}$ be $(k+2,l+2)$-free, and let $1\le i<j \le m$. Then there exist $x \in [k], y\in [l]$, such that
\begin{equation}\label{e:a_b}
 \alpha^i(x)>f(i,j)>\beta^i(y) \text{ and } \alpha^j(x)\le f(i,j)\le\beta^j(y).
\end{equation}

Symmetrically, there exist $x' \in [k], y'\in [l]$, s.t. 
\begin{equation}
\gamma^j(x')<f(i,j)<\delta^j(y') \text{ and } \gamma^i(x')\ge f(i,j)\ge\delta^i(y').
\end{equation}
\end{lemma}

\begin{proof} By symmetry it suffices to prove (\ref{e:a_b}).
Let $x$ be chosen minimum  $\alpha^{i}(x) > f(i,j)$. Note that $x \leq k$, as $\alpha^{i}(k+1)=+\infty$.  We have $\alpha^{i}(x~-~1)\leq f(i,j)$, or $x=1$. Therefore, if $x \geq 2$ then there exists an $x$-cup $(q_1,\ldots,q_{x-1},i)$ such that $f(q_{x-1},i)  \leq f(i,j)$. It follows that 
 $(q_1,\ldots,q_{x-1},i,j)$ is an $(x+1)$-cup.
Thus $\alpha^{i}(x)  \geq f(i,j)$.

Similarly, if $y$ is chosen maximum such that $\beta^{i}(y)< f(i,j)$ then  $f(i,j)\le\beta^j(y)$.
\end{proof}

The next corollary follows immediately from Lemma~\ref{lemma:a_b}.

\begin{corollary} \label{cor:neq2}
Let $f:[m]^{(2)} \to \mathbb{R}$ be $(k+2,l+2)$-free, and let $1\le i<j \le m$. Then $L_i \ne L_j$ and $R_i \ne R_j$. 
\end{corollary}

Note that Theorem~\ref{thm:cupsandcaps} immediately follows from Corollary~\ref{cor:neq2}, as $|\mc{R}^{k,l}|=\binom{k+l}{k}$.
The bijective proof of Theorem~\ref{thm:cupsandcaps} presented in this section is due to the first author~\cite{S-MO}. It is inspired by the beautiful argument used by Seidenberg~\cite{S}  to give a bijective proof of the  Erd\H{o}s-Szekeres theorem on monotone subsequences of a sequence~\cite{ES35}. A different bijective proof was recently discovered by Moshkovitz and Shapira~\cite{MS12}.

\section{Proof of Theorem~\ref{thm:main}}\label{sec:points}

We apply the framework introduced in the previous section to the Erd\H{o}s-Szekeres problem as follows.

Let $P=\{p_1, p_2, \ldots,p_m\}$ be a set of points in the plane in general position with Cartesian coordinates $(x_1,y_1),(x_2,y_2)\ldots,(x_m,y_m)$, respectively. 
Assume without loss of generality that $x_1<x_2<...<x_{m}$. We define a function $f_P: [m]^{(2)} \rightarrow \mathbb{R}$ as follows. For $1 \leq i < j \leq m$, let $f_P(i,j)$ be the slope of the line that passes through the points $i,j$, i.e. $$f_P(i,j)= \frac{y_j -y_i}{x_j - x_i}.$$

Assume that $P$ does not contain $n$ points
that are the vertices of a convex $n$-gon.
Thus $f_P$ is $(n,n)$-free, as every $n$-cup (or $n$-cap) with respect to $f_P$ indexes the vertices of a convex $n$-gon. In the following, for simplicity, we will refer to the sets of points indexed by $t$-cups (respectively, $t$-caps) with respect to $f_P$ as \emph{$t$-cups} (respectively, \emph{$t$-caps}) in $P$. We will refer to $3$-cups and $3$-caps as, simply, \emph{cups} and \emph{caps}.

We start by introducing two classical tools used to upper bound $ES(n)$ by T\'oth and Valtr~\cite{best-upper}, and  Erd\H{o}s and Szekeres~\cite{ES35}, respectively. 
The next theorem from \cite{best-upper}, allows us to further restrict potential functions $f_P$.

\begin{theorem} [\cite{best-upper}] \label{thm:best-upper}
Let $N$ be a positive integer such that every point set $P$ with $|P|\ge N$ contains either a set of $n$ points in convex position or a $(n-1)$-cap. Then $ES(n)\le N+1$. 
\end{theorem}

Theorem~\ref{thm:best-upper} allows us to consider only point sets $P$ such that $f_P$ is  $(n,n-1)$-free. 

The next lemma essentially appears in~\cite{ES35}, but we include its proof for completeness.

\begin{lemma} [\cite{ES35}] \label{ES} If $f_P$ is  $(n,n-1)$-free then no point in $P$ is simultaneously the end of an $(n-1)$-cup and the beginning of an $(n-2)$-cap. Symmetrically, no point $p$ is the end of an $(n-2)$-cap and the beginning of an $(n-1)$-cup. 
\end{lemma}

\begin{proof} Suppose for a contradiction that $(q_1,\ldots ,q_{n-2}, q)$ is an $(n-1)$-cup with respect to $f_P$ and that  $(q, q'_2, \ldots, q'_{n-2})$ is an $(n-2)$-cap. If $f(q_{n-2},q) \leq f(q, q'_2)$ then  $(q_1,\ldots ,q_{n-2}, q, q'_2)$ is an $n$-cup, and, otherwise, $(q_{n-2}, q,\ldots, q'_{n-2})$ is an $(n-1)$-cup, contradicting the fact that $f_P$ is  $(n,n-1)$-free.
\end{proof}

We'd like to extend the above lemma by showing that no point of $P$ can be  the end of an $(n-2)$-cup and the beginning of both an $(n-2)$-cap and $(n-1)$-cup. Unfortunately, this statement is not necessarily true. Fortunately, it can be made true by removing a negligible proportion of points in $P$.

Let $Q\subseteq P$ be the set of points which serve as an end of an $(n-2)$-cup and the beginning of both an $(n-2)$-cap and $(n-1)$-cup. For every point $q\in Q$, fix $V^q=(v^q_1,v^q_2,...,v^q_{n-3},q)$ an $(n-2)$-cup ending in $q$,  $U^q=(q,u^q_2,...,u^q_{n-2})$ an $(n-2)$-cap starting in $q$, and  $W^q:=(q,w^q_2,w^q_3,...,w^q_{n-1})$ an $(n-1)$-cup starting in $q$. (See Figure \ref{fig:example2}.) We say that $(V^q,U^q,W^q)$ is \emph{a $Q$-signature} of $q$.

\begin{figure}

\begin{center}
    \includegraphics[scale=0.65]{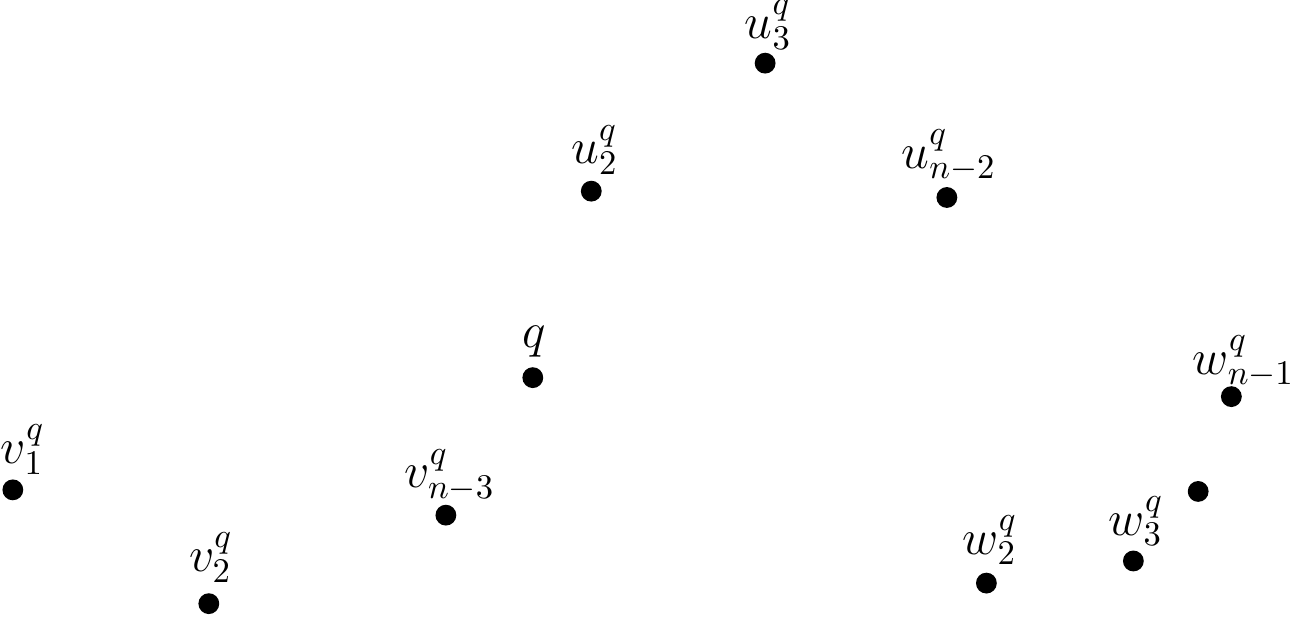}
\end{center} 

\caption{\label{fig:example2} A point $q\in Q$.}
\end{figure}

The following Claims~\ref{GV} and~\ref{lemmaQ'} were inspired by~\cite{V15}, yet they appear not to follow directly from any of the results in~\cite{V15}. 

\begin{claim}\label{GV} If $(V,U,W)$ is a $Q$-signature of a point $q \in Q$, such that $U^q=(q,u_2,...,u_{n-2})$, $W^q=(q,w_2,...,w_{n-1})$, then $u_2 = w_{n-1}$.
\end{claim}

\begin{proof}
Suppose for a contradiction that $u_2  \neq w_{n-1}$. Let $V=(v_1,v_2,...,v_{n-3},q)$ and let $\ell$ be the line passing through $v_{n-3}$ and $q$.

If $u_2$ is below the line $\ell$, then $v_{n-3},U$ is an $(n-1)$-cap. Therefore $u_2$ is above $\ell$. Similarly, if $w_2$ is above  $\ell$, then  $v_{n-3},W$ is an $n$-cup, therefore $w_2$ is below $\ell$.

Suppose that $u_2$ is to the left of $w_{n-1}$. If $q,u_2,w_{n-1}$ is a cup then $V,u_2,w_{n-1}$ is an $n$-cup. Otherwise, $W \cup \{u_2\}$ is a vertex set of a convex $n$-gon. (See Figure \ref{fig:example4}.)

\begin{figure}

\begin{center}
    \includegraphics[scale=0.55]{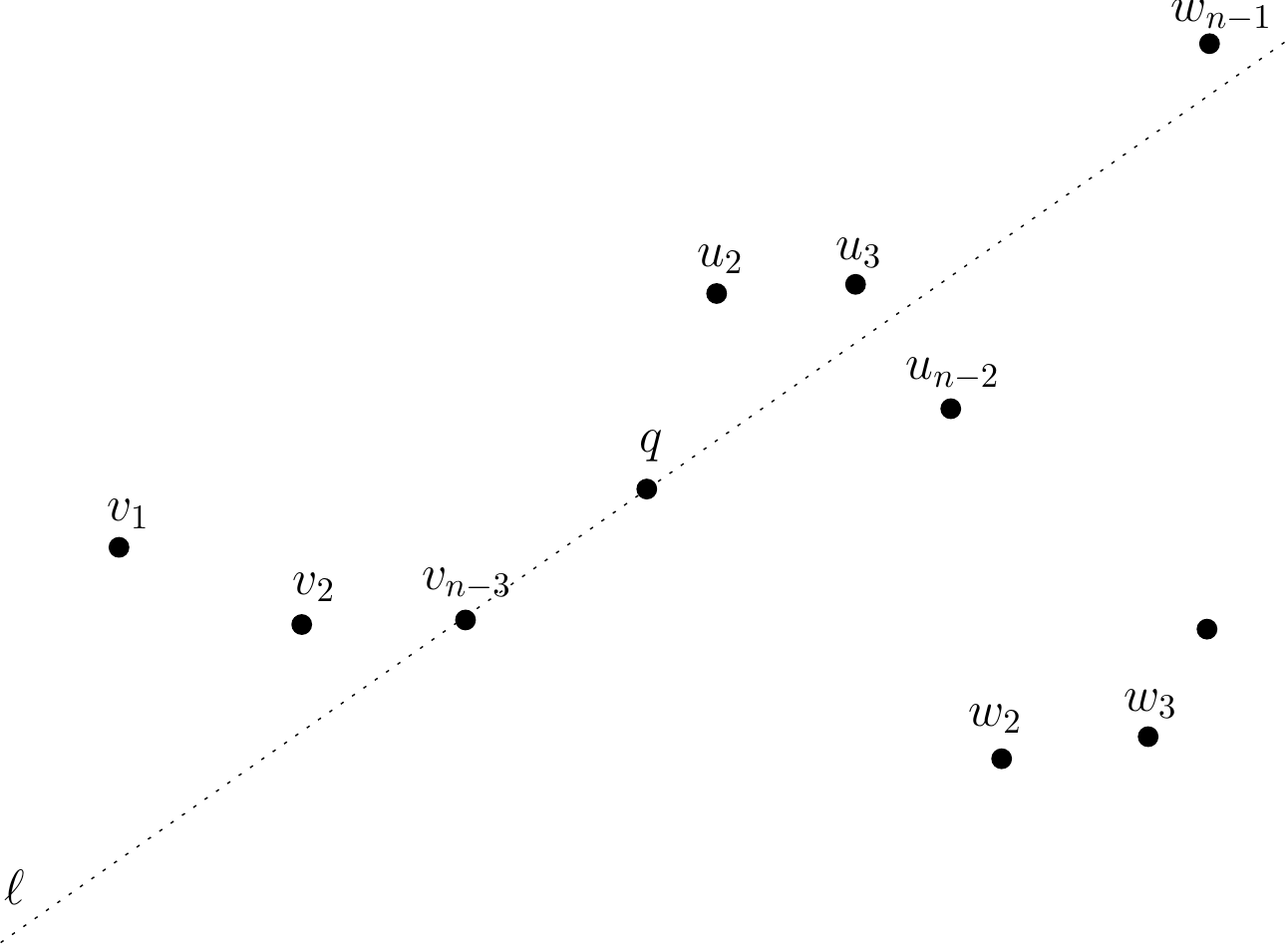}
\end{center} 

\caption{\label{fig:example4}  The point $u_2$ is to the left of $w_{n-1}$.}
\end{figure}

Therefore $u_2$ is to the right to $w_{n-1}$. Suppose that $w_{n-2},u_2,u_3$ is a cap.  Then $w_{n-3},w_{n-2},u_2$ is a cup, as otherwise $w_{n-3},w_{n-2},u_2, \ldots, u_{n-2}$ is an $(n-1)$-cap. If $q,w_{n-1},u_2$ is a cap, then  $q,w_{n-1},u_2,w_{n-2},w_{n-3},...,w_2$ is a vertex set convex  $n$-gon. Therefore $q,w_{n-1},u_2$ is a cup, implying that $w_{n-1},u_2,u_3$ is a cap. If $w_{n-2},w_{n-1},u_2$ is a cap, then $w_{n-2},w_{n-1},u_2,u_3,...,u_{n-2}$ is an $(n-1)$-cap. Thus $w_{n-2},w_{n-1},u_2$ is a cup, therefore $W,u_2$ is an $n$-cup.  (See Figure \ref{fig:example5}.) 

It remains to consider the subcase when $w_{n-2},u_2,u_3$ is  a cup. In this case $v_{n-3},q,w_{n-2}$ is a cup, and thus $V,w_{n-2},w_{n-1}$ is an $n$-cup, yielding the final contradiction.
\end{proof}

\begin{figure}

\begin{center}
    \includegraphics[scale=0.55]{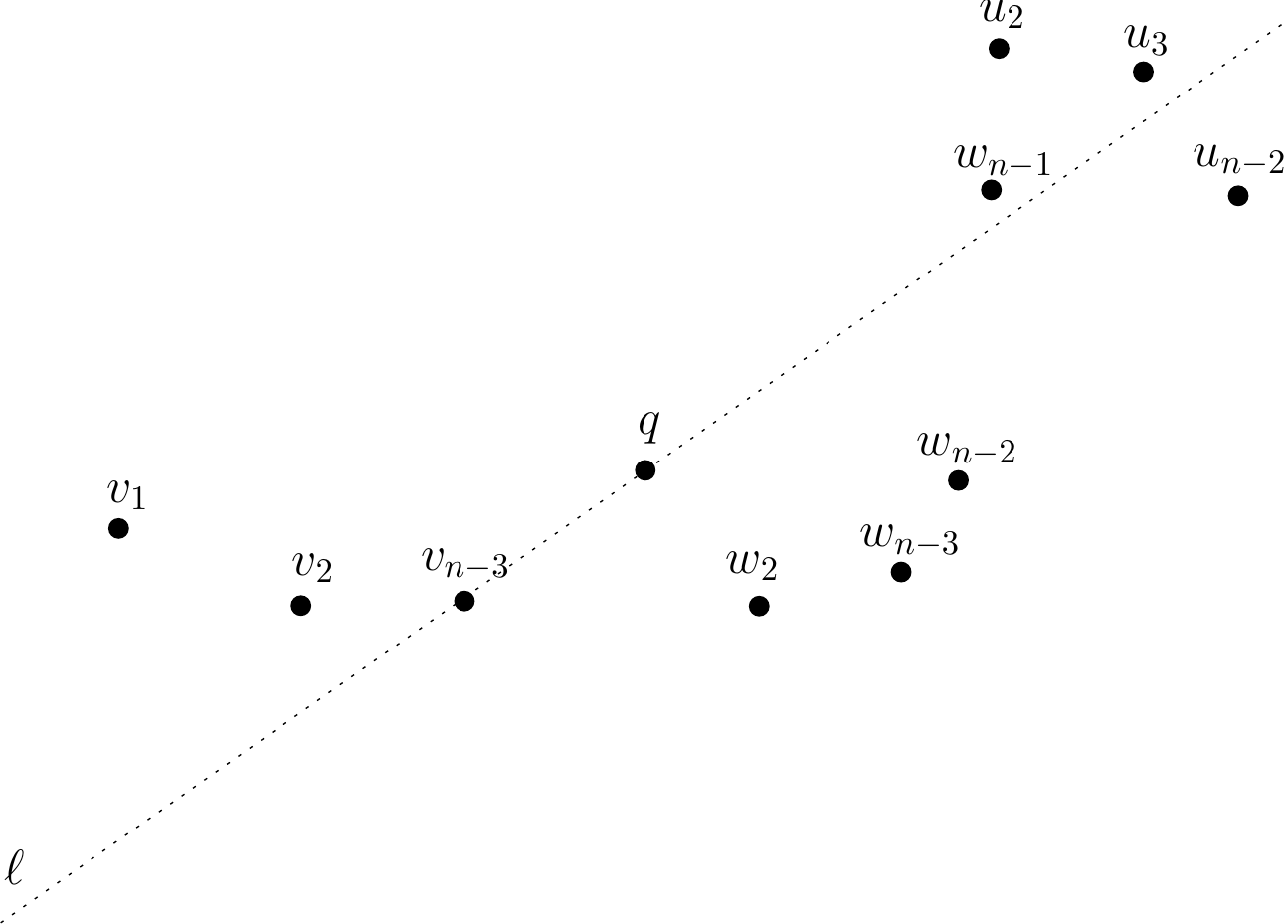}
\end{center} 

\caption{\label{fig:example5} The point $u_2$ is to the right of $w_{n-1}$.}
\end{figure}

We say that $q' \in P$ is \emph{a mate} of a point $q \in Q$ if there exist a $Q$-signature $(V,U,W)$ of $q$ with points of $U$ and $W$ labelled as in Claim~\ref{GV}, such that $q'=u_2=w_{n-1}$. Let $Q'$ be the set of all mates of points in $Q$. By Claim~\ref{GV} every point in $Q$ has at least one mate.

\begin{claim} \label{lemmaQ'}
$|Q'|\le |P|/(n-2)$.
\end{claim}

\begin{proof}
Consider $q'_1,q'_2\in Q'$. Let $q'_1$ and $q'_2$ be the mates of $q_1$ and $q_2$, respectively. Let $(V^{q_1},U^{q_1},W^{q_1})$, $(V^{q_2},U^{q_2},W^{q_2})$ be  $Q$-signatures of $q_1$ and $q_2$ with points labelled as in the definition such that $q'_1=u^{q_1}_2=w^{q_1}_{n-1}$ and $q'_2=u^{q_2}_2=w^{q_2}_{n-1}$. We assume without loss of generality that $q'_1$ is to the left of $q'_2$. We will show that $q_2$ is to the right of $q'_1$. This would imply the claim, as it will follow that the points of $W^{q_2}$ lie to the right of $q'_1$ and to the left of $q_2$, implying that the indices of any two points in $Q'$ differ by at least $n-2$.

Assume to the contrary that $q_2$ is to the left of $q'_1$. We claim that all the points of $P$ to the right of $q_2$ lie below the line that passes through $q_2$ and $q'_2$. Indeed, suppose that  $r$ is a point above this line. If $r$ is to the left of $q'_2$  then we find  a convex $n$-gon in $P$, using $W^{q_2}$ and $r$. If $r$ is to the right of $q'_2$ then we find a $n$-cup using $V^{q_2},q'_2$ and $r$. (Note that $V^{q_2},q'_2$ is an $(n-1)$-cup, because otherwise  $v^{q_2}_{n-3},U^{q_2}$ is an $(n-1)$-cap.). Analogously, all the points of $P$ to the right of $q_1$ must be below the line that goes through $q_1,q'_1$. (See Figure \ref{fig:example3} which illustrates a possible arrangement of points in this case.) Thus  $q_1,q'_1,q'_2,u^{q_2}_3,u^{q_2}_4,...,u^{q_2}_{n-2}$ is an $(n-1)$-cap, a contradiction. \end{proof}

\begin{figure}
\begin{center}
    \includegraphics[scale=0.7]{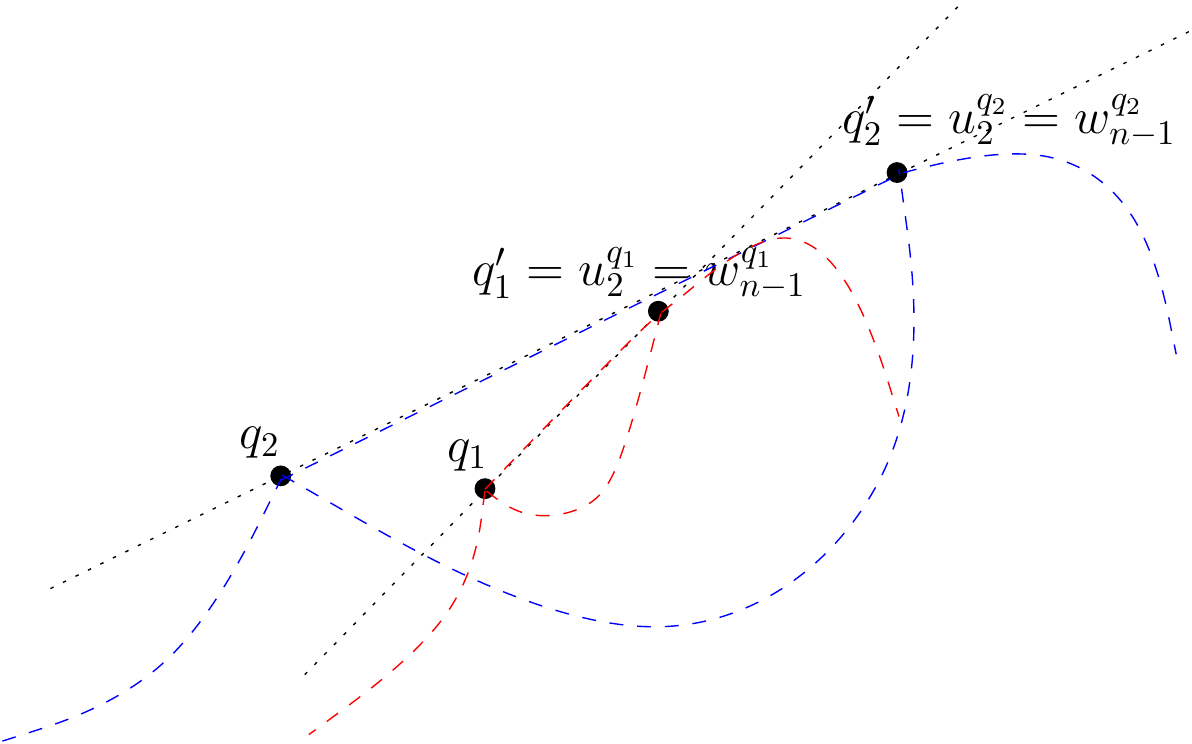}
\end{center} 
\caption{\label{fig:example3} A possible arrangement of points in Claim \ref{lemmaQ'}.}
\end{figure}

As our goal is to prove $m \leq (\frac{7}{8}+o(1))\binom{2n-5}{n-2}$, by Claim~\ref{lemmaQ'} we may consider $P - Q'$ instead of $P$, and thus assume that no point of $P$ is simultaneously an end of an $(n-2)$-cup and the beginning of both an $(n-2)$-cap and $(n-1)$-cup.

It is time to utilize the encoding of points introduced in Section~\ref{sec:sequences}.
Let $L_1,\ldots,L_{m} \in {\cal L}^{n-2,n-3}$ and  
$R_1,\ldots,R_{m} \in {\cal R}^{n-2,n-3}$ be the words associated with points of $P$ via the function $f_P$ defined in Section~\ref{sec:sequences}. We will write $\mc{L}$ and $\mc{R}$ instead of ${\cal L}^{n-2,n-3}$ and ${\cal R}^{n-2,n-3}$, for brevity.
For a pair of words $w_1$ and $w_2$, let ${\cal L}(w_1 * w_2)$ denote the set of words in ${\cal L}$ which start with $w_1$ and end with $w_2$. We write ${\cal L}(* w_2)$ and ${\cal L}(w_1*)$, instead of ${\cal L}(w_1 * w_2)$, if $w_1$ (respectively, $w_2$) is the null word. The sets ${\cal R}(w_1 * w_2)$ are defined analogously. The next claim follows immediately from the definitions of $L_i$ and $R_i$.

\begin{claim}\label{c:strings}
For $i \in [m]$ the following statements hold
\begin{itemize}
\item $p_i$ is the end of an $(n-1)$-cup if and only if $L_i \in \mc{L}(*\beta)$,
\item $p_i$ is the end of an $(n-2)$-cap if and only if $L_i \in \mc{L}(\alpha*)$,
\item $p_i$ is the beginning of an  $(n-1)$-cup if and only if $R_i \in \mc{R}(*\gamma)$,
\item $p_i$ is the beginning of an  $(n-2)$-cap if and only if $R_i \in \mc{R}(\delta*)$, and
\item $p_i$ is not an end of any $(n-2)$-cup if and only if $L_i \in \mc{L}(*\alpha\alpha)$.
\end{itemize}  
\end{claim}

Claim \ref{c:strings}, Lemma~\ref{ES} and our assumptions on $P$, based on Claims~\ref{GV} and~\ref{lemmaQ'}, imply the following.

\begin{claim}\label{c:last}
 If $R_i \in \mc{R}(\delta*\gamma)$ for some $i \in [m]$, then $L_i \in \mc{L}(\beta*\alpha\alpha)$.
\end{claim}

We are now ready to finish the proof of Theorem~\ref{thm:main}. Let $\mc{S}=\{R_1,\ldots,R_m\}$. By Claim~\ref{c:last} and Corollary~\ref{cor:neq2}, we have $$|\mc{S} \cap \mc{R}(\delta*\gamma)| \leq |\mc{L}(\beta*\alpha\alpha)|=\left(\frac{1}{8}+o(1)\right)\binom{2n-5}{n-2}.$$
Thus, by Corollary~\ref{cor:neq2},
$$m=|\mc{S}| \leq |\mc{S} \cap \mc{R}(\delta*\gamma)| + |\mc{R} -  \mc{R}(\delta*\gamma)| \leq \left(\frac{7}{8}+o(1)\right)\binom{2n-5}{n-2},$$
as desired.

\end{document}